\documentclass[12pt]{amsart}
\usepackage{amsfonts}
\usepackage{amssymb,latexsym}
\usepackage{enumerate}
\usepackage{mathrsfs}
\usepackage{amsmath}
\usepackage{amsthm}

scaled\magstephalf
\newcommand{\R}{\mathbb R}
\newtheorem{theorem}{Theorem}
\newtheorem{Thm}{Theorem}[section]

\newtheorem{Lem}[Thm]{Lemma}
\newtheorem{Rem}[Thm]{Remark}

\numberwithin{equation}{section}

\begin{document}
\setlength{\baselineskip}{1.2\baselineskip}
\title[Uniform Gradient Bounds and Convergence of Mean Curvature Flows]{%A Mean Curvature Flow with Prescribed Contact Angles in a High Dimensional Cylinder$^\S$
Uniform Gradient Bounds and Convergence of  Mean
Curvature Flows in a Cylinder$^\S$}
\thanks{$\S$ This research was partly supported by the NSFC (No. 12071299, 11601311).}
\author[Z. Gao, B. Lou and J. Xu]{Zhenghuan Gao$^{\dag}$, Bendong Lou$^\ddag$ and Jinju Xu$^{\ddag, *}$}
\thanks{$\dag$ Department of Mathematics, Shanghai University, Shanghai 200444, China}
\thanks{$\ddag$ Mathematics and Science College, Shanghai Normal University, Shanghai 200234, China.}
\thanks{$*$ Corresponding author.}
\thanks{{\bf Emails:} {\sf gzh2333@mail.ustc.edu.cn} (Z. Gao),  {\sf lou@shnu.edu.cn} (B. Lou),  {\sf jjxu@shnu.edu.cn} (J. Xu)}
\date{}

\begin{abstract}
We consider a mean curvature flow $V=H+A$ in a cylinder $\Omega\times \R$, where, $\Omega$ is a bounded domain in $\R^n$, $A$ is a constant driving force, $V$ and $H$ are the normal velocity and the mean curvature respectively of a moving hypersurface, which contacts the cylinder boundary with prescribed angle $\theta(x)$. Under certain conditions such as $\Omega$ is convex and $\|\cos\theta\|_{C^2}$ is small, or $\Omega$ is {\it non-convex} and $|A|$ is large, we derive the {\it uniform gradient bounds} for bounded and unbounded solutions (which is crucial in the study of the asymptotic behavior of the solutions). Then we present a trichotomy result on the convergence of the solutions as well as its criterion: when $A|\Omega|+\int_{\partial \Omega} \cos\theta(x) d\sigma>0$ (resp. $=0$, $<0$), the solution converges as $t\to \infty$ to a translating solution with positive speed (resp. stationary solution, a translating solution with negative speed).
\end{abstract}

\maketitle

\baselineskip 17pt

\section{Introduction}

Let $\Omega$ be a bounded domain containing the origin in $\R^n$ with $C^3$ boundary $\partial \Omega$, and let $\theta(x)$ be a $C^2$ function over $\overline{\Omega}$ with
$0<\theta (x)<\pi$. We consider a mean curvature flow in the cylinder $\Omega\times \R$ with prescribed contact angles $\theta(x)$ on $\partial \Omega\times \R$:
\begin{equation}\label{MCF}
\left\{
 \begin{array}{ll}
   V=H+A & \mbox{on } \Gamma_t \subset \Omega\times \R,\\
  {\bf n}\cdot (\gamma,0) = \cos \theta(x) & \mbox{on}\,\,  \overline{\Gamma}_t \cap (\partial \Omega\times \R),
  \end{array}
  \right.
\end{equation}
where $\Gamma_t\ (t\geq 0)$ is a family of hypersurfaces moving in the cylinder $\overline{\Omega}\times \R$, $V$
and $H$ denote the normal velocity and the mean curvature of $\Gamma_t$, respectively, $A$ is a real number denoting a driving force,
${\bf n}$ denotes the upward unit normal vector to $\Gamma_t$ and $\gamma$ denotes the inner unit normal to $\partial \Omega$. Then $H=-{\rm div} {\bf n}$.
In the case that $\Gamma_t$ is the graph of a function $y=u(x,t)$, the problem \eqref{MCF} can be converted into the following problem:
\begin{equation}\label{main-equ1}
\left\{
\begin{array}{ll}
\displaystyle u_t = a^{ij}(Du) D_{ij} u +A\sqrt{1+|Du|^2}, & (x,t)\in \Omega\times [0,\infty), \\
\displaystyle \frac{\partial u}{\partial \gamma}  =  -\cos\theta(x) \sqrt{1+|Du|^2}, &  (x,t)\in \partial \Omega \times [0, \infty),\\
 \displaystyle \frac{\partial u_0}{\partial \gamma} = -\cos\theta(x)\sqrt{1+|Du_0|^2}, & x\in \partial \Omega,
\end{array}
\right.
\end{equation}
where $Du=(D_1 u, \cdots, D_n u)$,
$$
a^{ij}(p) := \delta_{ij} - \frac{p_i p_j } {1+|p|^2} \quad\mbox{ for } p\in \R^n,
$$
and $u_0 \in C^\infty(\overline{\Omega})$ denotes the initial data.
The problem \eqref{MCF} or \eqref{main-equ1} describes the evolution of the graph of $u(\cdot,t)$
by its mean curvature (together with a driving force $A$) in the normal direction with prescribed contact angle $\theta$ on the boundary.

The problem \eqref{main-equ1} without driving force (i.e. $A=0$) has been studied by many authors since 1980s.
 To name only a few, Altschuler and Wu \cite{AW} considered the case $n=2$ and proved that when $\Omega$
 is strictly convex and $|D_\mathcal{T} \theta|\leq \min\kappa(x)$, where $\kappa(x)$  denotes the
 curvature of $\partial \Omega$ at $x$, the solution either converges as $t\to \infty$ to a minimal surface
  (when $\int_{\partial \Omega} \cos\theta (x) d\sigma =0)$, or to a translating solution. For higher dimensional cases,
   Huisken \cite{Huisken} considered the case $\theta(x)\equiv \frac{\pi}{2}$, and proved that the solution remains smooth,
   and converges as $t\to \infty$ to a minimal surface.
Gao, Ma, Wang and Weng \cite{GMWW} considered the problem with nearly vertical contact angle condition and proved
that when $\Omega$ is strictly convex, the solution  converges  to a translating solution.
Recently, Ma, Wang and Wei \cite{MWW} derived the uniform gradient estimates for the problem without the driving force and
with Neumann boundary condition $\frac{\partial u}{\partial \gamma}=\phi(x)$, and proved that when $\Omega$ is strictly convex,
the solution  converges  to a translating solution.
Wang, Wei and Xu \cite{WWX} generalized the results in \cite{MWW} to a kind of general capillarity-type boundary conditions
which does not include the prescribed contact angle case. In a recent long paper, Giga, Onoue and Takasao \cite{GOT}
 even derived a mean curvature flow with an interesting type of dynamic boundary conditions by taking limit for the Allen-Cahn equation.

Mean curvature flows with driving forces arise in the fields of geometry analysis (cf. \cite{Ang1, ChouZhu} etc.), and in the singular limit problems for reaction diffusion equations such as Allen-Cahn equation (cf. \cite{AHM,HKMN}, where the driving force describes the slight difference between two stable phases).
  The asymptotic behavior for the solutions to such equations has also been studied by some authors.
  For example, in case $n=1$, Angenent \cite{Ang1}, Chou and Zhu \cite{ChouZhu} studied the shrinking plane curves driven by $V=a({\bf n}) H + b({\bf n})$. Nara, Ninomiya and Taniguchi \cite{Nara-Tani, Ni-Tani} considered
  the existence and stability of the V-shaped translating solutions to $V=H+A$. Matano, Nakamura and Lou \cite{LMN, MNL}
  studied the translating solutions to the same equation in two dimensional periodic or almost periodic band domains.
In higher dimensional case, Guan \cite{Guan} considered the following general equation
\begin{equation}\label{guan-equ}
u_t = a^{ij}(Du) D_{ij} u -\psi(u,Du), \quad (x,t)\in \Omega\times [0,\infty).
\end{equation}
He gave {\it time-dependent gradient estimates} for the solutions. In the special case where $\psi(u, p) = ku\sqrt{1+|p|^2}$ with $ k > 0$,
or $\psi(u, p) = -n/u$ with $u > 0$, he showed that the $C^1$ estimates of the solutions are independent of the time (that is,
the solutions are bounded ones), and then proved the convergence of such solutions to stationary ones.

In this paper we study the asymptotic behavior for the solutions to \eqref{main-equ1}. Note that in our problem, the equation involves a driving force term, the cross section $\Omega$ of the cylinder can be convex or {\it non-convex}, and the solutions can be bounded or unbounded.

To study the quasilinear equation in \eqref{main-equ1}, a crucial step is to give a priori bounds like
$$
\|u\|_{C(\overline{\Omega}\times [0,T])} \leq C_0(T),\qquad
\|Du\|_{C(\overline{\Omega}\times [0,T])} \leq C_1(T).
$$
In many cases, when the $C^0$ bound is dependent on (resp. independent of) $T$, the gradient bound is also
dependent on (resp. independent of) $T$. To guarantee the well-posedness on the time interval $[0,T]$,
the time-dependent bounds $C_0 (T)$ and $C_1(T)$ are sufficient. However,
to study the asymptotic behavior for the solutions, one generally needs {\it uniform-in-time bounds}.
More precisely, in order to study the convergence of $u(x,t)-u(0,t)$ (no matter $u$ is bounded or not),
one usually establishes the $C^{2+\alpha}$ bounds by using the gradient bounds. The important thing is that,
 to derive a convergence result these bounds must be independent of the time, or uniform-in-time. In this sense,
  we say that the unform-in-time gradient bounds are of special importance in the qualitative study for \eqref{main-equ1},
  which are generally not easy, especially for unbounded solutions in high dimensional cylinder with a general cross section.
  The following theorem gives such a result for the problem \eqref{main-equ1}.

\begin{theorem}[{\bf Uniform-in-time gradient bound}]\label{thm:est}
Assume $\Omega\subset \R^n$ is a bounded domain with $\partial \Omega\in C^3$,
$\theta(x)\in C^2 (\overline{\Omega})$ with $0<\theta(x)<\pi$.
Let $u$ be a classical solution to \eqref{main-equ1} in the time interval $[0,T)$ for some $T\in (0,\infty]$. Then
there exists a positive number $C_1$ depending only on $u_0$ such that
$$
\sup\limits_{\Omega\times [0,T)} |u_t| \leq C_1.
$$
Moreover, there exists a positive number $C_2$ independent of $t$ and $T$ such that
\begin{equation}\label{C1-est}
\sup\limits_{\Omega\times [0,T)} |Du| \leq C_2,
\end{equation}
if one of the following conditions holds:
\begin{itemize}
\item[(i)] $n=2$ and the positive curvature $\kappa(x)$ of $\partial \Omega$ satisfies $ \kappa(x) > |A|+ \|D\theta \|_{C(\overline{\Omega})} $;
\item[(ii)] $n\geq2,\ \Omega$ is strictly convex and $\|\cos \theta\|_{C^2(\overline{\Omega})}\ll 1$;
\item[(iii)] $n\geq 2$ and $|A|\gg 1$.
\end{itemize}
\end{theorem}

\begin{Rem}\rm
Note that the gradient bound in \eqref{C1-est} is a {\it uniform-in-time} one,
which, as we have mentioned above, is crucial in the study of the convergence of the solutions. Such uniform gradient bounds were also obtained in
\cite{AW,GMWW,MWW} etc. for unbounded solutions to equations without driving force and in convex cylinders.
However, in (iii) of the previous theorem, we do not need the convexity for the cylinder. As far as we know,
there are no much results on the uniform gradient bounds for unbounded solutions in non-convex cylinders.
In \cite[Theorem 4]{Guan}, Guan considered a general equation \eqref{guan-equ} without the convexity assumption for $\Omega$, but the gradient
bounds he obtained for unbounded solutions are time-dependent ones, which are not applied to the
qualitative study for unbounded solutions (see (1.4) and the beginning of Section 3 in \cite{Guan}).
\end{Rem}

The gradient bound in the above theorem not only guarantees the global well-posedness, but also helps us to derive the convergence for the solutions. Denote
\begin{equation}\label{def:I}
I:= A|\Omega|+\int_{\partial \Omega} \cos\theta(x) d\sigma.
\end{equation}
According to the sign of $I$, we will prove a trichotomy result  on the asymptotic behavior of the solutions.

\begin{theorem}[{\bf Trichotomy and its criterion on the asymptotic behavior}]\label{thm:asy}
Assume the hypotheses in the previous theorem hold.
\begin{itemize}
\item[(i)] If $I=0$, then the problem \eqref{main-equ1} has a (unique up to a shift) stationary solution $V(x)$. Any solution $u$ to \eqref{main-equ1} converges as $t\to \infty$ to this stationary solution (with suitable shift);
\item[(ii)] If $I>0$, then the problem \eqref{main-equ1} has a (unique up to a shift) translating solution $\Phi(x)+ct$ with $c>0$. Any solution $u$ to the problem \eqref{main-equ1} converges as $t\to \infty$ to this translating solution (with suitable shift);
\item[(iii)] If $I<0$, then the problem \eqref{main-equ1} has a (unique up to a shift) translating solution $\Phi(x)+ct$ with $c<0$. Any solution $u$ to \eqref{main-equ1} converges as $t\to \infty$ to this translating solution (with suitable shift).
\end{itemize}
\end{theorem}

\begin{Rem}\rm
From the above theorems we see that (a) the cross section $\Omega$ of the cylinder in our problem can be a non-convex one;
(b) our equation involves a driving force term $A$; (c) the gradient bounds we obtained are uniform-in-time ones;
(d) the trichotomy of the asymptotic behavior for the solutions is characterized simply by the sign of $I$.
\end{Rem}

The proof of Theorem 1 needs very careful calculations and is rather technical, we postpone it to Section 3. In Section 2 we study the
asymptotic behavior for the solutions to \eqref{main-equ1}. In the first subsection we consider bounded solutions and their convergence to stationary ones,
in the second subsection we consider unbounded solutions and their convergence to translating solutions, and finally,
in the last subsection we complete the proof of Theorem \ref{thm:asy}.

\section{Asymptotic Behavior for the Solutions to \eqref{main-equ1}}

In this section we always assume that the hypotheses in Theorem \ref{thm:est} hold.
Then the problem \eqref{main-equ1} with any smooth initial data has a classical global solution, which has the uniform-in-time gradient bound as in \eqref{C1-est}.
We will study the asymptotic behavior for such solutions and prove Theorem \ref{thm:asy}.

Clearly, there are exactly three possibilities for the solutions to \eqref{main-equ1}:
\begin{itemize}
\item[({\bf A1})] all the global solutions to \eqref{main-equ1} are bounded;
\item[({\bf B1})] there is a global solution to \eqref{main-equ1} which is unbounded from above;
\item[({\bf C1})] there is a global solution to \eqref{main-equ1} which is unbounded from below.
\end{itemize}
We will show that these possibilities correspond respectively to the following statements:
\begin{itemize}
\item[({\bf A2})] the problem \eqref{main-equ1} has a unique (up to  a shift) stationary solution $V(x)$;
\item[({\bf B2})] the problem \eqref{main-equ1} has a unique (up to  a shift) translating solution $\Phi(x)+ct$ with $c>0$;
\item[({\bf C2})] the problem \eqref{main-equ1} has a unique (up to  a shift) translating solution $\Phi(x) +ct$ with $c<0$.
\end{itemize}
Moreover, each possibility can be characterized simply by the sign of $I$.

In the first subsection we consider bounded solutions, in the second subsection we consider unbounded solutions, and in the last subsection we complete the proof for Theorem \ref{thm:asy}.

\subsection{Bounded solutions converge to stationary solutions}

We first prove a lemma.

\begin{Lem}\label{lem:ss}
Assume the hypotheses in Theorem \ref{thm:est} hold. If the problem \eqref{main-equ1} has a bounded classical solution $u$,
then \eqref{main-equ1} has a unique (up to  a shift) stationary solution $V(x)$.
\end{Lem}

\begin{proof} We first use the Lyapunov functional to derive a quasi-convergence result, that is, any $\omega$-limit of $u(\cdot,t)$ is a stationary solution to \eqref{main-equ1}. Precisely, for any $\psi\in H^1(\Omega)$, define
$$
E[\psi] := \int_{\Omega} \Big( \sqrt{1+|D\psi|^2} -A\psi\Big) dx + \int_{\partial \Omega} \psi(x) \cos\theta(x) d \sigma.
$$
By our assumption, $u$ is a bounded classical solution to \eqref{main-equ1}. So $E[u(\cdot,t)]$ is well-defined and bounded from below for all $t>0$.
A direct calculation shows that
$$
\frac{d}{dt} E[u(\cdot,t)] = - \int_{\Omega} \frac{u^2_t}{\sqrt{1+|Du|^2}} dx \leq 0.
$$
Using this fact one can show in a standard way
that
$$
\sup\limits_{t\geq 1}\|u(\cdot,t)\|_{H^2(\Omega)}\leq C,\qquad
\lim\limits_{t\to \infty} \|u_t(\cdot,t)\|_{L^2(\Omega)}=0,
$$
for some $C>0$. Thus, $u$ has $\omega$-limits in $H^1(\Omega)$.
Suppose that $V(x)$ is one of them, then
$$
\|u(\cdot,t_n)-V(x)\|_{H^1(\Omega)} \to 0,\quad n\to \infty
$$
for some time sequence $\{t_n\}$. Taking $t=t_n$ in the
following equality
$$
\int_\Omega \frac{u_t}{\sqrt{1+|Du|^2}} \rho(x) dx
= \int_\Omega  \left[ -\frac{D\rho \cdot Du}{\sqrt{1+|Du|^2}} + A\rho \right] dx,\qquad \rho\in H^1_0(\Omega),
$$
and then sending $n\to \infty$ we have
$$
0 = \int_\Omega \left[ -\frac{D\rho \cdot DV}{\sqrt{1+|DV|^2}} + A\rho \right] dx .
$$
This implies that $V\in H^1(\Omega)$ is weak stationary solution. Finally, by the regularity theory of elliptic equations, we see that $V$ is actually a classical stationary solution. This proves the quasi-convergence result.

We now show the uniqueness (up to  a shift) for the stationary solution.
By contradiction, we assume that both $V_1(x)$ and $V_2(x)$ are stationary solutions to \eqref{main-equ1}.
Since the equation as well as the boundary condition in \eqref{main-equ1} is invariant in vertical spatial shift,
 $V_1(x)+h$, for any $h\in \R$, is also a stationary solution. Now we take a suitable $h$ such that
$$
V_2(x) \leq V_1(x)+h,\quad x\in \overline{\Omega},
$$
and the \lq\lq equality" holds at some point $x_0\in \overline{\Omega}$. Then the maximum principle implies that $V_2(x)\equiv V_1(x)+h$. This proves the uniqueness.
\end{proof}

For convenience, in the rest of the paper, we denote by $V(x)$ the unique stationary solution to \eqref{main-equ1} satisfying $V(0)=0$.

Now we can prove the main result for bounded solutions.

\begin{Thm}\label{thm:I=0}
Assume the hypotheses in Theorem \ref{thm:est} hold. Then the statements {\bf (A1)}, {\bf (A2)} and {\bf (A3)} are equivalent, where

\noindent
{\bf (A3)} any solution $u$ to \eqref{main-equ1} converges as $t\to \infty$ to a stationary solution.

Moreover, in any of the above cases, $I=0$.
\end{Thm}

\begin{proof}
{\bf (A1)} $\Rightarrow$ {\bf (A2)}. This is proved in the previous lemma.

{\bf (A2)} $\Rightarrow$ {\bf (A1)}. For any initial function $u_0$ we take a large $h>0$ such that
$$
V(x)-h \leq u_0(x) \leq V(x)+h,\quad x\in \overline{\Omega}.
$$
By comparison, $u(x,t;u_0)$, the solution to \eqref{main-equ1} with initial data $u_0$, remains bounded between $V+h$ and $V-h$. This proves (A1).

{\bf (A1)} $\Rightarrow$ {\bf (A3)}. Let $u$ be a bounded solution to \eqref{main-equ1}.
From the proof of the previous lemma we see that any $\omega$-limit of $u(\cdot,t)$ is $V(\cdot)+h$
 for some $h\in \R$. To prove {\bf (A3)} we only need to show that such $h$ is unique.
  By contradiction, we assume both $V(x)+h_1$ and $V(x)+h_2$ with $h_1<h_2$ are $\omega$-limits of $u(\cdot,t)$. Then there exists sufficiently large $T$ such that
$$
\|u(\cdot,T) - V(\cdot)-h_1\|_{C(\overline{\Omega})} < \frac{h_2 - h_1}{2}.
$$
Hence, $u(x,T)$ is smaller than the stationary solution $V(x)+ \frac{h_1+h_2}{2}$, so is $u(x,t)$ for all $t>T$ by comparison.
This, however, contradicts the assumption that $V(x)+h_2$ is also an
$\omega$-limit of $u$.

{\bf (A3)} $\Rightarrow$ {\bf (A1)}. This is obvious.

Finally, we prove $I=0$. Denote by $H(V)$ and ${\bf n}$ the mean curvature and the upward unit normal vector of the graph of $V(x)$, respectively. Then
$$
{\bf n} = \frac{(-DV, 1)}{\sqrt{1+|DV|^2}} \quad
\mbox{and}\quad
H(V) = - {\rm div}_{(x,y)} {\bf n} = {\rm div}_x  \frac{DV}{\sqrt{1+|DV|^2}}.
$$
Integrating the equation of $V$: $0 = H(V) +A $ over $\overline{\Omega}$ and using the boundary condition in \eqref{main-equ1} we have
$$
0= \int_{\Omega}  {\rm div}_x  \frac{DV}{\sqrt{1+|DV|^2}} dx +A|\Omega| = I:= \int_{\partial \Omega} \cos\theta(x) d\sigma +A|\Omega|.
$$
This proves the theorem.
\end{proof}

\subsection{Unbounded solutions converge to translating solutions}

In this subsection we consider the asymptotic behavior for unbounded solutions.
Since the discussion for solutions unbounded from below is similar as that for solutions unbounded
from above. We only study the latter case.

\begin{Lem}\label{lem:to-TW}
Assume the hypotheses in Theorem \ref{thm:est} hold. If the problem \eqref{main-equ1} has a classical global solution $u$ which is unbounded from above,
then \eqref{main-equ1} has a unique (up to a spatial shift) translating solution moving upward.
\end{Lem}

\begin{proof}
We divide the proof into four steps.

{\it Step 1. Convergence to $+\infty$}.  By Theorem \ref{thm:est}, $|Du|$ has a uniform-in-time bound $C$.

Since $u$ is assumed to be unbounded from above, there exists $T>0$ such that
$$
u(x,T) \geq u(x,0)+1 =u_0(x) +1,\quad x\in \overline{\Omega}.
$$
By the comparison principle we have
\begin{equation}\label{weak-increase}
u(x,t+T)\geq u(x,t)+1,\quad x\in \overline{\Omega},\ t\geq 0.
\end{equation}
In particular, for any positive integer $k$, we have
$$
u(x, t+kT) \geq u(x,t)+k,\quad x\in \overline{\Omega},\ t\geq 0.
$$
Together with the boundedness of $u_t$ in Theorem \ref{thm:est} we conclude that $u$ actually tends to $+\infty$ as $t\to \infty$.

{\it Step 2. Construction of an entire solution}. Set
\begin{equation}\label{def-uk}
u_k(x,t) := u(x,t+kT)-u(0,kT),\quad x\in \overline\Omega,\ t>-kT.
\end{equation}
Then we have
\begin{equation}\label{uk-est}
|u_{kt}(x,t)|,\ |Du_k(x,t)| \leq C,\quad x\in \overline\Omega,\ t>-kT,
\end{equation}
by the a priori bounds in Theorem \ref{thm:est}. Thus, for any $T_1 >0$
and any $\alpha \in (0,1)$, when $k\gg 1$, by the parabolic theory we have
$$
\|u_k \|_{C^{2+\alpha, 1+\alpha/2}(\overline\Omega\times [-T_1, T_1])} \leq C_1
$$
for some positive $C_1$ depending on $T_1$ but independent of $k$. Then, there is a subsequence
$\{u_{k_j}\}$ of $\{u_k\}$ and a function $U_1 (x,t)\in C^{2+\alpha, 1+\alpha/2}(\overline\Omega\times [-T_1, T_1])$ such
that $u_{k_j}\to U_1$ as $j\to \infty$ in $C^{2, 1}(\overline\Omega\times [-T_1, T_1])$ norm.
Using the Cantor's diagonal argument, we can find a subsequence of $\{u_{k_j}\}$, denoted it again by $\{u_{k_j}\}$, and
 a function $U\in C^{2+\alpha, 1+\alpha/2} (\overline\Omega\times \R)$ such that
$$
u_{k_j}(x,t)\to U(x,t) \mbox{ as } j\to \infty,\quad \mbox{in the } C^{2, 1}_{loc} (\overline\Omega \times \R) \mbox{ topology}.
$$
Clearly, $U$ is an entire solution to \eqref{main-equ1}, that is, a solution defined for all $t\in \R$.
Moreover, by \eqref{uk-est} and \eqref{weak-increase} we have
\begin{equation}\label{U-est}
|U_t (x,t)|,\ |D U (x,t)| \leq C \mbox{ and } U(x,t+T) \geq U(x,t) +1,\quad x\in \overline\Omega,\ t\in \R.
\end{equation}

{\it Step 3. Uniqueness (up to  a shift) of the entire solution}. We show that the entire solution $U$ is unique
up to  a shift in the sense that, for any entire solution $W(x,t)$ to \eqref{main-equ1} satisfying similar
estimates as in \eqref{U-est}, there exists a unique $l\in \R$ such that
\begin{equation}\label{shift}
W(x,t) \equiv U(x,t)+l.
\end{equation}

We apply a similar approach as that in \cite{MNL}.
For simplicity, we use the notation $w_1 \lessapprox w_2$ if $w_1(x)\leq w_2(x)$ and the \lq\lq equality" holds at some points in $\overline\Omega$.
Define
$$
L(t) := \max\{l \in \R\mid U(x,t)+l \leq W(x,t),\ x\in \overline\Omega\}
$$
and
$$
D (t) := \min\{\bar{d} \in \R\mid W(x,t) \leq U(x,t)+ L(t)+\tilde{d},\ x\in \overline\Omega\}.
$$
Then,
\begin{equation}\label{U<=W<=U}
U(x,t)+L(t) \lessapprox W(x,t) \lessapprox U(x,t)+L(t)+D(t).
\end{equation}
Moreover, we have the following properties:

\noindent
{\bf Claims}. (1) $D(t)\geq 0$ for all $t\in \R$;

(2) if $D(t_0) = 0$ for some $t_0\in \R$, then $L(t)=L(t_0)$ and $D(t)=0$ for all $t>t_0$;

(3) if $D(t_0)>0$ for some $t_0\in \R$, then for all $t<t_0$, $L(t)$ is strictly increasing,
$L(t)+D(t)$ and $D(t)$ are strictly decreasing. Both $D$ and $L$ are bounded in $t<t_0$.

We only prove the boundedness for $L$ and $D$ in Claim (3) since other conclusions follow from the maximum principle easily.
Since $|DU|, |DW|\leq C$ by our assumption, we see that the oscillations of $U$ and $W$ are bounded by
$C\cdot \phi(\Omega)$, where $\phi(\Omega)$ denotes the diameter of $\Omega$. Hence, by \eqref{U<=W<=U} we have
$$
D(t)\leq 3 C\cdot \phi(\Omega),\quad t\in \R.
$$
Since $L(t)$ is strictly increasing and $D(t)+L(t)$ is strictly decreasing in $t<t_0$ we have
$$
L(t_0) > L(t) > [D(t_0)+L(t_0)] -D(t) \geq [D(t_0)+L(t_0)] - 3 C\cdot \phi(\Omega),\quad t<t_0.
$$
This proves the boundedness for $L$ and $D$. Hence, in case (3) holds, there exist $\mathcal{D}>0,\ \mathcal{L}\in \R$ such that
\begin{equation}\label{limits-of-DH}
D(t)\to \mathcal{D},\quad L(t)\to \mathcal{L} \mbox{\ \ as  } t\to -\infty.
\end{equation}

Clearly, in order to prove \eqref{shift} we only need to show that $D(t)\equiv 0$. Assume by contradiction that $D(t_0)>0$, then Claim (3) holds.
 For all $x\in \overline\Omega,\ t\in \R$ and $k\in \mathbb{N}$, set
$$
U_k (x,t):= U(x, t-k)-U(0,-k) \mbox{ and } W_k (x,t) := W(x,t-k)- U(0,-k).
$$
Since $U_k (0,0)=0$ and $W_k (0,0)\in [L(-k), L(-k)+D(-k)]$ are bounded, we have similar $C^1$ bounds for
$U_k$ and $W_k$ as above. Hence, we can take limit as $k\to \infty$ (if necessary, we take a subsequence) to obtain
$$
U_k \to \mathcal{U} \mbox{ and } W_k \to \mathcal{W}\ \  \mbox{ as }k\to \infty,\quad \mbox{ in the topology of } C^{2,1}_{loc} (\overline\Omega \times \R),
$$
for some entire solutions $\mathcal{U}$ and $\mathcal{W}$. By \eqref{U<=W<=U} and \eqref{limits-of-DH} we then have
$$
\mathcal{U}(x,t) +\mathcal{L} \lessapprox \mathcal{W} (x,t) \lessapprox \mathcal{U}(x,t) +\mathcal{L}+\mathcal{D}.
$$
This, however, contradicts Claim (3) for the entire solutions $\mathcal{U}$ and $\mathcal{W}$ instead of the entire solutions $U$ and $W$.  Therefore, $D(t)\equiv 0$ holds, and so $L(t)\equiv l\in \R$. This proves \eqref{shift}.

In what follows, we use $U(x,t)$ to denote the unique entire solution satisfying the normalized condition $U(0,0)=0$. (For, otherwise, we just use $U(x,t)-U(0,0)$ to replace the original $U$).
By the definition we see that $u_k$ also satisfies the normalized condition $u_k(0,0)=0$.
Hence, due to the uniqueness of $U$, the whole sequence $\{u_k\}$ converges to $U$:
$$
u_k(x,t) = u(x,t+kT)-u(0,kT)\to U(x,t),\quad \mbox{as} \ k\to \infty.
$$
Since $T$ can be any large number by Step 1, we actually have
\begin{equation}
 u(x,t+s)-u(0,s)\to U(x,t),\quad \mbox{as} \ s\to \infty,
\end{equation}
that is, $u(x,t)$ converges to the unique entire solution in a moving frame.

{\it Step 4. The existence and uniqueness of the translating solution}. Let $U(x,t)$ be the unique  entire solution obtained above.
 Then, for any $\tau \in \R$, $U(x,t+\tau)$ is also an entire solution to \eqref{main-equ1}.
 We can regard this solution as $W$ in Step 3 to conclude that, there exists a unique $l=l(\tau )$ such that $U(x,t+\tau)\equiv U(x,t)+l(\tau)$. Therefore,
$$
\lim\limits_{\tau \to 0} \frac{l(\tau )}{\tau} = \lim\limits_{\tau\to 0} \frac{U(x,t+\tau)-U(x,t)}{\tau} = c:= U_t(x,t).
$$
Hence $U(x,t)$ is nothing but a translating solution with the form $\Phi(x) + ct$ for some function $\Phi(x)$ ($\Phi$ satisfies $\Phi(0)=0$ by the normalized condition).
The sign $c>0$ can be determined directly by the second inequality in \eqref{U-est}.
 The uniqueness (up to a spatial shift) of this translating solution follows from the uniqueness of the entire solution in Step 3.
\end{proof}

Now we prove some equivalent statements for solutions unbounded from above.

\begin{Thm}\label{thm:unbounded-from-above}
Assume the hypotheses in Theorem \ref{thm:est} hold. Then the statements {\bf (B1)}, {\bf (B2)} and {\bf (B3)} are equivalent, where

\noindent
{\bf (B3)} any solution $u$ to \eqref{main-equ1} converges as $t\to \infty$ in a moving frame to an upward moving translating solution.

Moreover, in any of the above cases, $I>0$.
\end{Thm}

\begin{proof}
{\bf (B1)} $\Rightarrow$ {\bf (B2)}. This is proved in the previous lemma.

{\bf (B2)} $\Rightarrow$ {\bf (B3)}. Let $\Phi(x) +ct$ with $\Phi(0)=0,\ c>0$ be a translating solution. For any initial function $u_0$ we take a large $h>0$ such that
$$
u_0(x) \geq \Phi(x) -h ,\quad x\in \Omega.
$$
By comparison we know that $u(x,t;u_0) \geq \Phi(x)+ct -h$, and so $u(x,t;u_0)$ is a solution unbounded from above. Using Lemma \ref{lem:to-TW} and its proof we see that $u$ converges in a  moving frame to the translating solution, that is,
$$
u(x,t+s;u_0)- u(0,s;u_0) \to \Phi(x) + ct,\quad \mbox{as }\ s\to \infty.
$$

{\bf (B3)} $\Rightarrow$ {\bf (B1)}. This is obvious.

Finally, we prove $I>0$. The equation for $\Phi(x) +ct$ is
$$
c= a^{ij}(D\Phi) D_{ij} \Phi +A\sqrt{1+|D\Phi|^2},
$$
or, equivalently,
$$
\frac{c}{\sqrt{1+|D\Phi|^2}} = {\rm div}_x  \frac{D\Phi}{\sqrt{1+|D\Phi|^2}} +A.
$$
Integrating this equation over $\Omega$ and using the fact $c>0$ and the boundary condition in \eqref{main-equ1} we have
$$
0 < \int_{\Omega}  {\rm div}_x  \frac{D\Phi}{\sqrt{1+|D\Phi|^2}} dx +A|\Omega| = \int_{\partial \Omega} \cos\theta(x) d\sigma +A|\Omega|=I.
$$
This proves the theorem.
\end{proof}

\subsection{Proof of Theorem \ref{thm:asy}}

Now we can complete the proof of Theorem \ref{thm:asy}.
Note that the statements {\bf (A1)}, {\bf (B1)} and
{\bf (C1)} give all the possible cases and they are independent to each other.
When {\bf (B1)} holds, it follows from Theorem \ref{thm:unbounded-from-above} that $I>0$.
When {\bf (C1)} holds we have $I<0$ in a similar way. Therefore, when $I=0$, the only possible
case is {\bf (A1)}, and so the conclusions in Theorem \ref{thm:asy} (i) follow from Theorem
\ref{thm:I=0}. Similarly, when $I>0$, only the case {\bf (B1)} is possible and so
Theorem \ref{thm:asy} (ii) follows from Theorem \ref{thm:unbounded-from-above}. This completes the proof of Theorem \ref{thm:asy}. \qed

\section{A priori Bounds}
Let  $u\in C^{3,2}(\overline{\Omega}\times [0,T])$ be a classical solution to \eqref{main-equ1} in the time interval $[0,T]$ for some $T\in (0,\infty]$.
In this section we will derive some a priori bounds for $u_t$ and $Du$. The bounds are uniform ones, that is, they are independent of $t$ and $T$.
For simplicity, we use the following notations:
$$
 u_i := D_i u, \quad  u_{ij}:= D_{ij} u,\quad v :=\sqrt{1+|Du|^2}.
$$
In addition, we will follow the Einstein's sum convention: all the repeated indices denote summation from $1$ to $n$.

On the contact angles, we assume $\theta(x)\in C^2(\overline{\Omega})$ with $0<\theta(x) <\pi$, and so
\begin{equation}
S_0  := \max\limits_{\overline{\Omega}} |\cos\theta(x)| <1.
\end{equation}

In the first subsection we derive a bound for $u_t$, and in the rest parts we derive the bounds for $Du$,
including the case $\Omega\subset \R^2$ being convex in Subsection 3.2,
the case $\Omega\subset \R^n$ being convex in Subsection 3.3, and the case $\Omega\subset \R^n$ being a general domain in Subsection 3.4.
Since non-convex results are rarely seen, we give a more detailed proof of this situation.
In the paper, for convenience, we will use subscript indices to denote derivatives for functions and
follow the summation convention.

\subsection{A priori bound for $u_t$}

\begin{Lem}\label{Lem1.1}
Let $u\in C^{3,2}(\overline{\Omega}\times [0,T])$ (for some $T>0$) be a solution to \eqref{main-equ1}. Then
\begin{align*}
\max_{\overline\Omega\times[0,T]}|u_t|=\max_{\overline\Omega}|u_t(\cdot,0)|.
\end{align*}
\end{Lem}

\begin{proof} Following the method in \cite[Lemma 2.2]{AW}, it suffices to prove the following fact:
For any fixed $T_0\in (0,T]$, if $u_t$ admits a positive
local maximum at some point $(x_0,t_0)\in\overline\Omega\times[0,T_0]$, that is,
\begin{align*}
u_t(x_0,t_0)= \max_{\overline\Omega\times[0,T_0]}u_t\ge 0,
\end{align*}
then $t_0=0$.

Suppose by contradiction that $t_0>0$. It is easy to calculate that $u_t$ satisfies the equation
\begin{align}
\frac{d}{dt}u_{t}=& a^{ij}(u_t)_{ij}
-\frac{2}{v}a^{ij}v_{i}u_{tj}+\frac{A}{v}u_ju_{tj}.\label{1.1}
\end{align}
This implies that $u_t$ satisfies the parabolic maximum principle, and so $x_0$ can be taken on $\partial\Omega$.
Choose the coordinates in $\mathbb{R}^n$ such that the positive $x_n$-axis is
the interior normal direction to $\partial \Omega $ at $x_0$. Then at the point $(x_0, t_0)$,
\begin{align}\label{1.2}
u_{tk}=&0,\qquad k=1,\ldots,n-1,
\end{align}
and
\begin{align}\label{1.3}
v_t=\frac{u_ku_{kt}}{v}=\frac{u_nu_{nt}}{v}=-\cos\theta\cdot u_{nt}.
\end{align}
Differentiating the  boundary condition in $t$, we have
\begin{align}\label{1.4}
u_{nt}(x_0,t_0)=& -\cos\theta(x_0)\cdot v_t (x_0, t_0)=\cos^2\theta(x_0)\cdot  u_{nt}(x_0,t_0).
\end{align}
Therefore,
\begin{align}\label{1.5}
u_{nt}(x_0,t_0)=&0.
\end{align}
This however, contradicts the Hopf lemma at $(x_0,t_0)$.
\end{proof}

\subsection{Uniform $|Du|$-bounds when $\Omega\subset \R^2$ is convex}

\begin{Thm}\label{Thm1.1}
Assume $\Omega\subset \R^2$ is a strictly convex bounded domain with $C^3$ boundary $\partial \Omega$.
Let $u\in C^{3,2}(\overline\Omega\times[0,T])$ be a solution to the  problem \eqref{main-equ1} in the time interval $[0,T]$ for some $T\in (0,\infty]$.
If the curvature $\kappa(x)$ of $\partial\Omega$ satisfies
$$
\kappa(x) -|A|-\max\limits_{\overline{\Omega}} |D\theta| \ge\delta_1,
$$
for some $\delta_1 >0$, then
$$
\sup\limits_{\Omega\times [0,T]} |Du|\leq C_1(\Omega, S_0 , \|\theta\|_{C^1 (\overline{\Omega})}, \delta_1,u_0).
$$
\end{Thm}

\begin{proof}
We first show that the maximum of $|Du|^2$ is attained on $\Omega\times\{0\}$ or on $\partial\Omega\times[0,T]$. A direct calculation shows that
\begin{align}\label{Thm1.1-|Du|^2-t}
\begin{split}
(|Du|^2)_t=&2u_ku_{kt}.
\end{split}\end{align}
From the  first equation in \eqref{main-equ1}, we have
\begin{align}\label{utk-1}\begin{split}
2u_ku_{tk}=&2a^{ij}u_{ijk}u_k+2a^{ij}_{p_m}u_{mk}u_{ij}u_k+2Au_kv_k\\
=&a^{ij}(|Du|^2)_{ij}-2a^{ij}u_{ik}u_{jk}+2a^{ij}_{p_m}u_{mk}u_{ij}u_k+\frac{Au_k}{v}(|Du|^2)_k,
\end{split}\end{align}
where
\begin{align}\label{aij-pm}
a^{ij}_{p_m}=&\frac{2u_iu_ju_m}{v^4}-\frac{\delta_{im}u_j+\delta_{jm}u_i}{v^2}.
\end{align}
Substituting \eqref{aij-pm} into \eqref{utk-1} we have
\begin{align}\label{utk-2}\begin{split}
2u_ku_{tk}
=&a^{ij}(|Du|^2)_{ij}-2a^{ij}u_{ik}u_{jk}-\frac{4}{v^2}a^{ij}u_{kj}u_{im}u_{m}u_k+\frac{Au_k}{v}(|Du|^2)_k\\
=&a^{ij}|Du|^2_{ij}-2a^{ij}u_{ik}u_{jk}-\frac{1}{v^2}a^{ij}(|Du|^2)_i(|Du|^2)_j+\frac{Au_k}{v}(|Du|^2)_k.
\end{split}\end{align}
Hence \eqref{Thm1.1-|Du|^2-t} becomes
\begin{align}\label{Lem2.1-v-t-2}\begin{split}
(|Du|^2)_t
=&a^{ij}|Du|^2_{ij}-2a^{ij}u_{ik}u_{jk}-\frac{1}{v^2}a^{ij}(|Du|^2)_i(|Du|^2)_j+\frac{Au_k}{v}(|Du|^2)_k.
\end{split}\end{align}
Since $a^{ij}$ is a semi-positive definite matrix, we have $a^{ij}u_{ik}u_{jk}\ge0$.  It follows that
\begin{align}\label{Lem2.1-v-t-3}\begin{split}
(|Du|^2)_t
\le&a^{ij}|Du|^2_{ij}-\frac{1}{v^2}a^{ij}(|Du|^2)_i(|Du|^2)_j+\frac{Au_k}{v}(|Du|^2)_k.
\end{split}\end{align}
Now we can use the maximum principle to conclude that the maximum of $|Du|^2$ is attained on $\Omega\times\{0\}$ or $\partial\Omega\times[0,T]$.

Next we assume that $|Du|^2$ attains its maximum at the point $(x_0,t_0)\in\partial\Omega\times[0,T]$.
Following the idea of Altschuler and Wu \cite{AW}, denote the unit inner normal vector and the unit
 tangential vector by $\mathcal{N}$ and $\mathcal{T}$, respectively.
On $\partial\Omega$, we define $a^{\mathcal{T}\mathcal{N}}=a^{ij}\mathcal{T}_i\mathcal{N}_j$, $a^{\mathcal{T}\mathcal{T}}=a^{ij}\mathcal{T}_i\mathcal{T}_j$, $a^{\mathcal{N}\mathcal{N}}=a^{ij}\mathcal{N}_i\mathcal{N}_j$.
As did in \cite{AW}, we can give a smooth extension of $\mathcal{N}$ and $\mathcal{T}$ to a thin neighborhood of $\partial\Omega$.
Consider the set of coordinate vectors $\{\frac{\partial}{\partial r},\frac{\partial}{\partial\theta}\}$,
 $\frac{\partial}{\partial r}$ is orthogonal to the level sets of the distance function $d(x)=\mathrm{dist}(x,\partial\Omega)$
 such that $|\frac{\partial}{\partial r}|^2=1$ and $\frac{\partial}{\partial r}\mid_{\partial \Omega}=\mathcal{N}$, $\frac{\partial}{\partial\theta}$
 is taken such that $\langle\frac{\partial}{\partial r}, \frac{\partial}{\partial\theta}\rangle=0$ and $\frac{\partial}{\partial\theta}\mid_{\partial\Omega}=\mathcal{T}$.
 Let $\phi$ be a function such that $|\phi^{-1}\frac{\partial}{\partial\theta}|^2=1$. Now we can extend $\mathcal{N}$ and $\mathcal{T}$ by $\{\frac{\partial}{\partial r}, \phi^{-1}\frac{\partial}{\partial\theta}\}$. By Lemma 2.1 in \cite{AW}, we have
\begin{align}\label{NT}\begin{split}
\nabla_\mathcal{T}\mathcal{T}=&\kappa \mathcal{N},\,\,\nabla_\mathcal{T}\mathcal{N}=-\kappa \mathcal{T},\,\,\nabla_\mathcal{N}\mathcal{T}=0,\,\,\nabla_\mathcal{N}\mathcal{N}=0,\\
u_{\mathcal{T}\mathcal{N}}=&u_{\mathcal{N}\mathcal{T}}+\kappa u_{\mathcal{T}}.
\end{split}\end{align}
We shall deal with the second order derivatives $u_{\mathcal{N}\mathcal{N}}$ and $u_{\mathcal{T}\mathcal{N}}$.
For the term $u_{\mathcal{T}\mathcal{N}}$, we will use the boundary condition to obtain the following equalities.
\begin{align}
u_\mathcal{N}=&-\cos\theta \cdot v,\label{equalities-1.1}\\
u^2_\mathcal{N}=&\cos^2\theta \cdot v^2,\label{equalities-1.2}\\
u_\mathcal{T}^2=&\sin^2\theta \cdot v^2-1.\label{equalities-1.3}
\end{align}
Assume, without loss of generality, that $|u_\mathcal{T} (x_0,t_0)|\ge1$. Otherwise, by the boundary condition we obtain the bound
\begin{align}\label{Lem2.1-v-t-3}\begin{split}
|Du|^2 (x_0,t_0) <\frac{2}{1-S^2_0}-1.
\end{split}\end{align}
Hence, at $(x_0,t_0)$,
\begin{equation}\label{|Du|^2T}
(|Du|^2)_\mathcal{T}= 0=v_\mathcal{T},\quad u_\mathcal{N}u_{\mathcal{N}\mathcal{T}}+u_\mathcal{T}u_{\mathcal{T}\mathcal{T}}=0,
\end{equation}
and
\begin{equation}\label{|Du|^2N-1}
(|Du|^2)_\mathcal{N}\le0,\quad  u_\mathcal{N}u_{\mathcal{N}\mathcal{N}}+u_\mathcal{T}u_{\mathcal{T}\mathcal{N}}\le0.
\end{equation}
Differentiating the equalities \eqref{equalities-1.1}-\eqref{equalities-1.3} in the tangential direction, we have
\begin{align}
u_{\mathcal{N}\mathcal{T}}=&\sin\theta\theta_\mathcal{T}v-\cos\theta v_\mathcal{T},\label{equalities-2.2}\\
u_{\mathcal{T}\mathcal{N}}=&\sin\theta\theta_\mathcal{T}v-\cos\theta v_\mathcal{T}+\kappa u_\mathcal{T},\label{equalities-2.3}\\
u_{\mathcal{T}\mathcal{T}}=&\cos\theta\sin\theta\theta_\mathcal{T}\frac{v^2}{u_\mathcal{T}}+\frac{\sin^2\theta vv_\mathcal{T}}{u_\mathcal{T}}.\label{equalities-2.4}
\end{align}
In particular, at $(x_0,t_0)$, by \eqref{|Du|^2T} we have
\begin{align}
u_{\mathcal{N}\mathcal{T}}=&\sin\theta\theta_\mathcal{T}v,\label{equalities-2.2'}\\
u_{\mathcal{T}\mathcal{N}}=&\sin\theta\theta_\mathcal{T}v+\kappa u_\mathcal{T},\label{equalities-2.3'}\\
u_{\mathcal{T}\mathcal{T}}=&\cos\theta\sin\theta\theta_\mathcal{T}\frac{v^2}{u_\mathcal{T}}.\label{equalities-2.4'}
\end{align}
For the term $u_{\mathcal{N}\mathcal{N}}$, we will rewrite $u_t$ in the equation \eqref{main-equ1} under the $\mathcal{N},\mathcal{T}$ directions.
\begin{align}\label{derivatives}\begin{split}
\partial_i=&\mathcal{T}^i\partial_\mathcal{T}+\mathcal{N}^i\partial_\mathcal{N},\\
\partial_{ij}=&(\mathcal{T}^j\partial_\mathcal{T}+\mathcal{N}^j\partial_\mathcal{N})(\mathcal{T}^i\partial_\mathcal{T}+\mathcal{N}^i\partial_\mathcal{N})\\
=&\mathcal{T}^j\mathcal{T}^i\partial_\mathcal{T}\partial_{\mathcal{T}}
+\mathcal{T}^j\mathcal{N}^i\partial_\mathcal{T}\partial_\mathcal{N}+\mathcal{N}^j\mathcal{T}^i\partial_\mathcal{N}\partial_\mathcal{T}
+\mathcal{N}^j\mathcal{N}^i\partial_\mathcal{N}\partial_\mathcal{N}
+\kappa \mathcal{T}^j\mathcal{N}^i\partial_\mathcal{T}-\kappa \mathcal{T}^j\mathcal{T}^i\partial_\mathcal{N}.
\end{split}\end{align}
Since
\begin{align}\label{aTN}\begin{split}
a^{\mathcal{T}\mathcal{T}}=&a^{ij}\mathcal{T}^i\mathcal{T}^j=\frac{1+u_\mathcal{N}^2}{v^2},\\
a^{\mathcal{T}\mathcal{N}}=&a^{ij}\mathcal{T}^i\mathcal{N}^j=-\frac{u_\mathcal{T}u_\mathcal{N}}{v^2},\\
a^{\mathcal{N}\mathcal{N}}=&a^{ij}\mathcal{N}^i\mathcal{N}^j=\frac{1+u_\mathcal{T}^2}{v^2}=\sin^2\theta,
\end{split}\end{align}
 by \eqref{equalities-2.2'}-\eqref{equalities-2.4'}, we can rewrite $u_t$ at $(x_0,t_0)$ as the following.
\begin{align}\label{ut}\begin{split}
u_t=&a^{ij}u_{ij}+Av\\
=&a^{\mathcal{T}\mathcal{T}}u_{\mathcal{T}\mathcal{T}}+a^{\mathcal{N}\mathcal{T}}u_{\mathcal{N}\mathcal{T}}
+a^{\mathcal{T}\mathcal{N}}u_{\mathcal{T}\mathcal{N}}+a^{\mathcal{N}\mathcal{N}}u_{\mathcal{N}\mathcal{N}}
+\kappa a^{\mathcal{N}\mathcal{T}}u_{\mathcal{T}}-\kappa a^{\mathcal{T}\mathcal{T}}u_\mathcal{N}+Av\\
=&a^{\mathcal{T}\mathcal{T}}u_{\mathcal{T}\mathcal{T}}+2a^{\mathcal{N}\mathcal{T}}u_{\mathcal{N}\mathcal{T}}+a^{\mathcal{N}\mathcal{N}}u_{\mathcal{N}\mathcal{N}}
+2\kappa a^{\mathcal{N}\mathcal{T}}u_{\mathcal{T}}-\kappa a^{\mathcal{T}\mathcal{T}}u_\mathcal{N}+Av\\
=&\sin^2\theta u_{\mathcal{N}\mathcal{N}}+\frac{1+u_\mathcal{N}^2}{v^2}\cos\theta\sin\theta\theta_\mathcal{T}\frac{v^2}{u_\mathcal{T}}
-2\frac{u_\mathcal{T}u_\mathcal{N}}{v^2}\sin\theta\theta_\mathcal{T}v\\&-2\kappa \frac{u_\mathcal{T}u_\mathcal{N}}{v^2}u_{\mathcal{T}}-\kappa \frac{1+u_\mathcal{N}^2}{v^2}u_\mathcal{N}+Av\\
=&\sin^2\theta u_{\mathcal{N}\mathcal{N}}+\frac{2u^2_\mathcal{T}+1+u_\mathcal{N}^2}{u_\mathcal{T}}\cos\theta\sin\theta\theta_\mathcal{T}+\kappa \frac{2u^2_\mathcal{T}+1+u_\mathcal{N}^2}{v}\cos\theta+Av\\
=&\sin^2\theta u_{\mathcal{N}\mathcal{N}}+\frac{u^2_\mathcal{T}+v^2}{u_\mathcal{T}}\cos\theta\sin\theta\theta_\mathcal{T}+\kappa \frac{u^2_\mathcal{T}+v^2}{v}\cos\theta+Av.
\end{split}\end{align}
Hence, we have
\begin{align}\label{uNN}\begin{split}
\sin^2\theta u_{\mathcal{N}\mathcal{N}}=&u_t-Av-\kappa \cos\theta\frac{v^2+u_\mathcal{T}^2}{v}-\frac{u^2_\mathcal{T}+v^2}{u_\mathcal{T}}\cos\theta\sin\theta\theta_\mathcal{T}.
\end{split}\end{align}
 Multiplying $\sin^2\theta$ at both sides of \eqref{|Du|^2N-1}, and using \eqref{equalities-2.3'} and \eqref{uNN}, we have
\begin{align}\label{|Du|^2N-2}\begin{split}
0\ge& u_\mathcal{N}\sin^2\theta u_{\mathcal{N}\mathcal{N}}+u_\mathcal{T}\sin^2\theta u_{\mathcal{T}\mathcal{N}}\\
=&-\cos\theta v \big[u_t-Av
-\kappa \cos\theta\frac{v^2+u_\mathcal{T}^2}{v}-\frac{u^2_\mathcal{T}+v^2}{u_\mathcal{T}}\cos\theta\sin\theta\theta_\mathcal{T}\big]\\
&+u_\mathcal{T}\sin^2\theta \big[\sin\theta\theta_\mathcal{T}v+\kappa u_\mathcal{T}\big]\\
=&-\cos\theta u_tv+A\cos\theta v^2
+\kappa  (\cos^2\theta v^2+u_\mathcal{T}^2)+\sin\theta\theta_\mathcal{T}\frac{v(v^2-1)}{u_\mathcal{T}}\\
=&\kappa |Du|^2+A\cos\theta v^2-\cos\theta u_tv+\frac{v\theta_\mathcal{T}}{\sin\theta u_\mathcal{T}}\big[u^2_\mathcal{T}+\cos^2\theta\big]\\
\ge&(\kappa-|A\cos\theta|-|\theta_\mathcal{T}|)v^2-|u_t|v-\frac{v|\theta_\mathcal{T}|}{\sin\theta }-\kappa.
\end{split}\end{align}
By our assumption $\kappa-|A|-|\theta_\mathcal{T}|\ge\delta_1>0$, we obtain
$$\sup_{\Omega\times [0,T]}|Du|\leq C_1 (\Omega, S_0 ,\|\theta\|_{C^1(\overline{\Omega})},\delta_1,u_0).
$$
This proves the theorem.
\end{proof}

\subsection{Uniform $|Du|$-bounds when $\Omega\subset \R^n$ is convex}

When $\Omega$ is a strictly convex smooth domain in high dimensions,  we firstly state the result as follows.

\begin{Thm}\label{Thm1.2}
Assume $\Omega\subset \mathbb{R}^n\ (n\geq 2)$ is a strictly convex bounded domain with $\partial \Omega \in C^3$.
 Let $u\in C^{3,2}(\overline\Omega\times[0,T])$ be a solution to the problem \eqref{main-equ1}
 in the time interval $[0,T]$ for some $T\in (0,\infty]$.
  If $\theta\in C^2(\overline{\Omega})$ and
\begin{equation}\label{angle perturbation0}
\|\cos\theta\|_{C^2(\overline{\Omega})} \ll 1,
\end{equation}
then
\begin{equation}\nonumber
\sup_{\Omega\times[0,T]} |Du|\leq C_2 (n,\Omega, u_0, \|\cos\theta\|_{C^2(\overline{\Omega})}).
\end{equation}
\end{Thm}

\begin{Rem} \rm The  proof of Theorem~\ref{Thm1.2} is based on the maximum principle method by choosing a suitable auxiliary function.
Though many auxiliary functions can be seen  to get the gradient bound for such problems in the article or the references therein,
the key point for us is to get the uniform  gradient bound (independent of the time $t$). Ma, Wang and Wei in \cite{MWW}
initially gave  the uniform  gradient bound for Neumann boundary value problem, which depends on the strict convexity of the domain $\Omega$. Here
we follow the idea of \cite{MWW} by choosing the  auxiliary function (slightly different from theirs)
$$\varphi(x,t):= \log w+\alpha \tilde{h},$$
where $w :=v-u_k\tilde{h}_k\cos\theta$, $\alpha$ is a positive constant and $\tilde{h}$
 is a function defined over $\overline{\Omega}$. As the detailed proof is similar to theirs or can be seen in our arXiv: 2210.16475 version,  we omit it here.
\end{Rem}

\subsection{Uniform $|Du|$-bounds when $\Omega\subset \R^n$ is a general domain}

Let $\Omega$ be a bounded smooth domain in $\mathbb{R}^n$ with $\partial \Omega\in C^{3}$.
Recall that the distance function
\begin{align*}
 d(x):=\texttt{dist}(x,\partial \Omega)
 \end{align*}
 is smooth near $\partial\Omega$ and $Dd=\gamma$ on $\partial\Omega$.
Moreover, we assume that $d$ is
extended to be a smooth function over $\overline{\Omega}$ satisfying $d\ge 0$ and $|Dd|\le 1$ in $\overline{\Omega}$.

Our next theorem gives a uniform gradient bound for the solutions to \eqref{main-equ1} when
$\Omega$ is not a convex domain. Though its proof is also to use the maximum principle, the
auxiliary function we will construct is different from the previous one and so
the proofs are independent of each other. Compared to the convex case, the non-convex one is more novel. Hence  we will give a more detailed proof in the following.

\begin{Thm}\label{Thm1.3}
Assume $\Omega\subset \mathbb{R}^n\ (n\geq 2)$ is a bounded domain (does not have to be convex) with $\partial \Omega \in C^3$.
Let $u\in C^{3,2}(\overline\Omega\times[0,T])$   be a solution to the problem \eqref{main-equ1}
in the time interval $[0,T]$ for some $T\in (0,\infty]$.
If $\theta\in C^2(\overline{\Omega})$ and
$|A|\gg 1$, then
\begin{equation}\nonumber
\sup_{\Omega\times[0,T]} |Du|\leq C_3 (n,\Omega, u_0, S_0, \|\theta\|_{C^2(\overline{\Omega})}).
\end{equation}
\end{Thm}

\begin{proof}
Choosing the auxiliary function
$$
\varphi(x,t):= \log w+\alpha_0 d,
$$
where $w :=v+u_kd_k\cos\theta$ and $\alpha_0$ is a positive constant to be determined later.
We will derive the uniform bound for $|Du|$.

As above, we assume $\varphi(x,t)$ attains its maximum at $(x_0,t_0)\in \overline{\Omega}_T :=\overline{\Omega}\times [0,T]$, and divide the proof into three cases.

\textbf{Case (i)}: $x_0\in\partial\Omega$.

At $(x_0,t_0)$, we choose the coordinate in $\mathbb{R}^n$ so that the positive $x_n$-axis is the inner normal direction to $\partial\Omega$,
 which is exactly equal to $\gamma$.
We denote by $D'u=(u_1,\ldots,u_{n-1})$ and $u_n$ the tangential and normal part of $Du$ on the boundary by the choice of the coordinate.
Let $\frac{\partial}{\partial x_{\beta}}$ be the tangential directions, $\beta=1,\cdots,n-1$. Let $\beta, \eta=1,\cdots, n-1$ be
numbered corresponding to the coordinates of the tangential space.

Then on the boundary $\partial\Omega$,
\begin{align}\begin{split}
d_\beta=&0,\quad d_n=1,\quad
 d_{\beta n}=0,\quad d_{\beta\eta}=0,\quad\beta,\eta=1,2,\cdots,n-1.\label{d-1}
\end{split}\end{align}

First by the  boundary condition in \eqref{main-equ1}, we have
\begin{align}
w=&v+u_{n} \cos\theta=v\sin^2\theta,\,\,\,\text{on}\,\,  \partial\Omega,\label{4w-1}\\
u_n^2=&\cot^2 \theta (1+|D' u|^2),\,\,\,\text{on}\,\,  \partial\Omega.\label{4un-1}
\end{align}
Then at $x_0\in\partial\Omega$, by the Hopf lemma and the boundary condition, we can compute directly to derive
\begin{equation}\label{4varphi-n-1}
\begin{split}
 0\geq \varphi_n
=&\frac{w_n }{w}+\alpha_0 \\
=&\frac{1}{w}\big[v_n+(u_kd_k)_n\cos\theta+(\cos\theta )_n u_kd_k\big]+\alpha_0\\
=&\frac{1}{w}\big[\frac{ u_ku_{kn}}{v}+u_{kn}d_k\cos\theta+u_kd_{kn}\cos\theta+u_n(\cos\theta )_n\big]+\alpha_0\\
=&\frac{1}{w}\big[\frac{1}{v}\sum\limits_{\beta=1}^{n-1} u_{\beta}u_{\beta n} +\frac{1}{v}u_nu_{nn}+u_{nn}\cos\theta+u_nd_{nn}\cos\theta+u_n(\cos\theta )_n\big]+\alpha_0\\
=&\frac{1}{w}\big[\frac{1}{v}\sum\limits_{\beta=1}^{n-1} u_{\beta}u_{\beta n} +u_nd_{nn}\cos\theta+u_n(\cos\theta )_n\big]+\alpha_0\\
=&\frac{1}{w}\big[\frac{1}{v}\sum\limits_{\beta=1}^{n-1} u_{\beta}u_{n\beta} +\frac{1}{v}\sum\limits_{\beta,\eta=1}^{n-1}\kappa_{\beta\eta}u_{\beta}u_{\eta} +u_nd_{nn}\cos\theta+u_n(\cos\theta )_n\big]+\alpha_0.
\end{split}
\end{equation}
By the definition of $w$, we obtain
\begin{equation}\label{4v-beta-1}
\begin{split}
0=&\varphi_{\beta}=\frac{w_\beta}{w}=\frac{[v+(u_kd_k)\cos\theta]_\beta}{w},\\
 0=&w_\beta=v_\beta+u_{n\beta}\cos\theta+u_n(\cos\theta)_\beta,
\end{split}
\end{equation}
which implies that
\begin{equation}\label{4vbeta-1}
\begin{split}
v_\beta=&-u_{n\beta}\cos\theta-u_n(\cos\theta)_\beta.
\end{split}
\end{equation}
Differentiating the boundary condition $u_n=-\cos\theta v$ along the tangential directions, then
\begin{equation}\label{4vbeta-2}\begin{split}
u_{n\beta}=&(-\cos\theta v)_\beta=-(\cos\theta)_{\beta} v-\cos\theta v_\beta.
\end{split}\end{equation}
From \eqref{4vbeta-1}, \eqref{4vbeta-2} and the boundary condition in \eqref{main-equ1}, we deduce that
\begin{equation}\label{4unbeta-1}\begin{split}
 u_{n\beta}=&-\csc^2\theta(\cos\theta)_{\beta} v+u_n\csc^2\theta\cos\theta (\cos\theta)_\beta=v\csc\theta(1+\cos^2\theta)\theta_\beta.
\end{split}\end{equation}
Inserting \eqref{4unbeta-1} into \eqref{4varphi-n-1}, we have
\begin{equation}\label{4varphi-n-3}
\begin{split}
 0\geq \varphi_n
=&\frac{1}{w}\big[\csc\theta(1+\cos^2\theta)\sum\limits_{\beta=1}^{n-1} u_{\beta}\theta_\beta
+\frac{1}{v}\sum\limits_{\beta,\eta=1}^{n-1}\kappa_{\beta\eta}u_{\beta}u_{\eta} \\&+u_nd_{nn}\cos\theta+(\cos\theta )_nu_n\big]+\alpha_0\\
=&\frac{\csc\theta(1+\cos^2\theta)}{w}\sum\limits_{\beta=1}^{n-1} u_{\beta}\theta_{\beta}
+
\frac{1}{wv}\sum\limits_{\beta,\eta=1}^{n-1}\kappa_{\beta\eta}u_{\beta}u_{\eta}-d_{nn}\cot^2\theta+\cot\theta\theta_n+\alpha_0\\
\ge&\alpha_0-\frac{2(n-1)|Du|}{w}\csc\theta \|\theta\|_{C^1(\partial\Omega)}-|\cot\theta||\theta_n|
-\frac{1}{wv}|\kappa_{0}||D'u|^2-|d_{nn}|\cot^2\theta\\
\ge&\alpha_0-\frac{(2n-1)}{\sin^3\theta} \|\theta\|_{C^1(\partial\Omega)}
-\frac{1}{wv}|\kappa_{0}||D'u|^2-|d_{nn}|\cot^2\theta\\
\ge&\alpha_0-\frac{(2n-1)}{\sin^3\theta} \|\theta\|_{C^1(\partial\Omega)}
-\frac{\kappa_{0}}{\sin^2\theta}-\frac{|D^2d|}{\sin^2\theta}\\
\ge&1,
\end{split}
\end{equation}
where we have chosen $\kappa_0>0$ satisfying $-\kappa_0\delta_{\beta\eta}\le\kappa_{\beta\eta}\le\kappa_0\delta_{\beta\eta}$ and
$$
\frac{\alpha_0}{2} =
1+\frac{2n-1}{\inf\limits_{\partial\Omega}\sin^3\theta} \|\theta\|_{C^1(\partial\Omega)}
+\frac{\kappa_{0}}{\inf\limits_{\partial\Omega}\sin^2\theta}+\frac{|D^2d|}{\inf\limits_{\partial\Omega}\sin^2\theta}.
$$
Therefore, $\varphi$ cannot attain its maximum on $\partial\Omega$.

\textbf{Case (ii)}: $x_0\in \Omega$ and $t_0=0$, then we have
\begin{equation}\nonumber
\varphi(x,t)\leq \varphi(x_0,0)=\log(\sqrt{1+|Du_0|^2}+(u_0)_{k}d_k\cos\theta)+\alpha_0 d\leq C(n,\Omega,u_0, S_0 ),
\end{equation}
and
\begin{equation}
\sup_{\Omega\times[0,T]} |Du|\leq C(n,\Omega,u_0, S_0 ).
\end{equation}

\textbf{Case (iii)}: $x_0\in \Omega$ and $T\geq t_0>0$. Then
at $(x_0,t_0)$,  we have
\begin{equation}\label{4varphii}
\begin{split}
0=&\varphi_i =\frac{w_i}{w}+\alpha_0 d_i,\\
0\le&\varphi_t =\frac{w_t}{w}.
\end{split}\end{equation}
Thus
\begin{equation}\label{4varphii=0-wi}
\begin{split}
w_i=&-\alpha_0 d_iw,\quad \text{for}\quad 1\leq i\leq n,
\end{split}\end{equation}
and
\begin{equation}\label{4aijvarphiij-varphit-1}
\begin{split}
0 \ge& a^{ij}\varphi_{ij}-\varphi_t\\
=& \frac{1}{w}(a^{ij}w_{ij}-w_t)-\frac{a^{ij}w_iw_j}{w^2}+\alpha_0  a^{ij}d_{ij}
\\
=&\frac{1}{w}(a^{ij}w_{ij}-w_t)+\alpha_0  a^{ij}d_{ij}- \alpha_0^2a^{ij}d_id_j.
\end{split}
\end{equation}
At $(x_0,t_0)$,  rotating the coordinate $(x_1,x_2,\ldots,x_n)$ such that
\begin{equation}\nonumber
|Du|=u_1>0,\quad\text{and}\quad \{u_{\alpha\beta}\}_{2\leq \alpha,\beta\leq n}\quad \text{is diagonal}.
\end{equation}
(If $u_1(x_0,t_0)=0$, then  the gradient bound is easily to be obtained.
As we can see $\forall  (x,t)\in \Omega\times (0,T]$, $\varphi(x,t)\le\varphi(x_0,t_0)=\alpha_0 d(x_0)$, we have
$$\log w+\alpha_0 d\le\alpha_0 d(x_0),\quad \text{and}\quad w\le e^{\alpha_0 d(x_0)}.$$
By the definition of $w$, we observe that $(1- S_0 )v\le w\le 2v$. So $v$ is bounded.)

Then it follows that at $(x_0,t_0)$,
\begin{equation}\label{4aij-3}
a^{11}=\frac{1}{v^2};\,\,a^{ii}=1,\quad\text{for}\quad i\geq 2;\,\,a^{ij}=0,\quad\text{for}\quad i\neq j.
\end{equation}
By \eqref{4varphii=0-wi}, for $i\ge1$ we have
\begin{equation}\label{4varphii=0-wi-2}\begin{split}
-\alpha_0 d_i w=w_i=&v_i+u_{ki}d_k\cos\theta+(d_k\cos\theta )_i u_k\\
= &\frac{u_ku_{ki}}{v}+\cos\theta d_ku_{ki}+u_1(d_1\cos\theta )_i\\
=&S_ku_{ki}+u_1(d_1\cos\theta )_i,
\end{split}\end{equation}
where we set
\begin{equation}\label{4S-k-1}
S_k :=\frac{u_k}{v}+\cos\theta d_k,\,\, 1\le k \le n.
\end{equation}
Then, by \eqref{4varphii=0-wi-2} we have
\begin{equation}\label{4skuki-1}\begin{split}
S_ku_{ki}=&-u_1(d_1\cos\theta )_i-w\alpha_0 d_i .
\end{split}\end{equation}
For $i=1$,  it follows that
\begin{equation}\label{4u11-1}
\begin{split}
S_ku_{k1}=&-u_1(d_1\cos\theta )_1- w\alpha_0 d_1,
\end{split}
\end{equation}
 and for  $i\ge 2$, we replace $i$ by $\beta\ge2$ to obtain
\begin{align}
S_ku_{k\beta}=&-u_1(d_1\cos\theta )_\beta -w\alpha_0 d_\beta ,\label{4Skukbeta-1}\\
S_1u_{1\beta}=&-S_\beta u_{\beta\beta}-u_1(d_{1}\cos\theta)_\beta -w\alpha_0 d_\beta.\label{4S1u1beta-1}
\end{align}

Next we shall calculate $\frac{1}{w}(a^{ij}w_{ij}-w_t)$. Firstly we differentiate $w$ to obtain
\begin{equation}\label{4wi-1}\begin{split}
w_i
= &\frac{u_ku_{ki}}{v}+u_{ki}d_k\cos\theta+u_k(d_k\cos\theta )_i=S_ku_{ki}+u_k(d_k\cos\theta )_i ,\\
w_t= &\frac{u_ku_{kt}}{v}+\cos\theta d_ku_{kt}=S_ku_{kt},
\end{split}\end{equation}
and
\begin{equation}\label{4wij-1}\begin{split}
w_{ij}=& \frac{S_ku_{kij}}{v}+\frac{u_{ki}u_{kj}}{v}-\frac{u_ku_{ki}u_mu_{mj}}{v^3}
+u_{kj}(d_k\cos\theta )_i +u_{ki}(d_k\cos\theta)_j+u_k(d_k\cos\theta )_{ij}
\\
=& S_ku_{kij}+\frac{u_{ki}u_{kj}}{v}-\frac{u_1^2u_{1i}u_{1j}}{v^3}
+u_{kj}(d_k\cos\theta )_i +u_{ki}(d_k\cos\theta)_j+u_1(d_1\cos\theta )_{ij} .
\end{split}\end{equation}
Thus from \eqref{4wi-1} and \eqref{4wij-1} we have
\begin{equation}\label{4aijwij-wt-1}\begin{split}
a^{ij}w_{ij}-w_t
=& S_k(a^{ij}u_{ijk}-u_{kt})+\frac{a^{ij}u_{ki}u_{kj}}{v}-\frac{u_1^2}{v^3}a^{ij}u_{1i}u_{1j}\\&
+2 a^{ij}(d_k\cos\theta )_i u_{kj}+u_1a^{ij}(d_1\cos\theta )_{ij}.
\end{split}\end{equation}
Since
\begin{equation}\label{reequ}
a^{ij}u_{ij}+Av=u_t,
\end{equation}
we have
\begin{equation}\label{u11v2-1}
\frac{u_{11}}{v^2}=-\sum_{\beta=2}^nu_{\beta\beta}+u_t-Av.
\end{equation}
Differentiating \eqref{reequ} in $x_k$ for $k\ge 1$,  we have
\begin{align}\label{reequ-diff-1}
a^{ij}u_{ijk}-u_{tk}=&-a^{ij}_{p_m}u_{mk}u_{ij}-\frac{Au_{1}}{v}u_{1k}.
\end{align}
Therefore, by \eqref{aij-pm}, \eqref{reequ-diff-1} becomes
\begin{align}\label{reequ-diff-2}\begin{split}
a^{ij}u_{ijk}-u_{tk}
=&\frac{2}{v^2}a^{ij}u_{kj}u_{im}u_{m}-\frac{Au_{1}}{v}u_{1k}\\
=&\frac{2u_{1}}{v^2}a^{ij}u_{kj}u_{1i}-\frac{Au_{1}}{v}u_{1k}.
\end{split}\end{align}
Inserting \eqref{reequ-diff-2}  into \eqref{4aijwij-wt-1},  we have
\begin{equation}\label{4aijwij-wt-2}\begin{split}
a^{ij}w_{ij}-w_t
=& \frac{2u_{1}}{v^2}a^{ij}S_ku_{kj}u_{1i}+\frac{a^{ij}u_{ki}u_{kj}}{v}-\frac{u_1^2}{v^3}a^{ij}u_{1i}u_{1j}
+2 a^{ij}(d_k\cos\theta )_i u_{kj}\\&-\frac{Au_{1}}{v}S_ku_{k1}
+u_1a^{ij}(d_1\cos\theta )_{ij}.
\end{split}\end{equation}
By \eqref{4skuki-1}, \eqref{4S1u1beta-1}  and \eqref{reequ}, we have
\begin{equation}\label{4aijwij-wt-3}\begin{split}
a^{ij}w_{ij}-w_t
=& \frac{1}{v}\sum_{\beta=2}^nu^2_{\beta \beta}+2\sum_{\beta=2}^n(d_\beta\cos\theta )_\beta u_{\beta\beta}-\frac{Au_{1}}{v}S_ku_{k1}+u_1a^{ij}(d_1\cos\theta )_{ij}\\&+\frac{2}{v^3}\sum_{\beta=2}^nu^2_{1\beta }+2\sum_{\beta=2}^n\big[(d_1\cos\theta )_\beta+\frac{(d_\beta\cos\theta )_1}{v^2}\big] u_{1\beta}\\&
+\frac{1}{v}(\frac{u_{11}}{v^2})^2+\big[\frac{2u_{1}}{v^2}S_ku_{k1}+2(d_1\cos\theta )_1\big]\frac{u_{11}}{v^2}
+\frac{2u_{1}}{v^2}\sum_{\beta=2}^nS_ku_{k\beta}u_{1\beta}.
\end{split}\end{equation}
In the following, we need to calculate the last three formulas in \eqref{4aijwij-wt-3}.
By \eqref{u11v2-1}, \eqref{4S1u1beta-1} and \eqref{4u11-1}, we have
\begin{equation}\label{4aijwij-wt-3-1}\begin{split}
 &\frac{1}{v}(\frac{u_{11}}{v^2})^2
 +\big[\frac{2u_{1}}{v^2}S_ku_{k1}+2(d_1\cos\theta )_1\big]\frac{u_{11}}{v^2}+\frac{2u_{1}}{v^2}\sum_{\beta=2}^nS_ku_{k\beta}u_{1\beta}
\\
=&\frac{1}{v}\big[\sum_{\beta=2}^nu_{\beta\beta}+Av-u_t\big]^2+\big[\frac{2u_{1}}{v^2}S_ku_{k1}+2(d_1\cos\theta )_1\big]\big[u_t-Av-\sum_{\beta=2}^nu_{\beta\beta}\big]\\&
-\frac{2u_{1}}{v^2}\sum_{\beta=2}^n\big[u_1(d_1\cos\theta )_\beta +w\alpha_0 d_\beta\big]u_{1\beta}\\
=&\frac{1}{v}\big[\sum_{\beta=2}^nu_{\beta\beta}\big]^2+\big[2A-\frac{2u_t}{v}
+\frac{2u_{1}w}{v^2}\alpha_0d_1-\frac{2}{v^2}(d_1\cos\theta )_1\big]\sum_{\beta=2}^nu_{\beta\beta}\\&
+\frac{2\alpha_0u_{1}w }{S_1v^2}\sum_{\beta=2}^nd_\beta S_\beta u_{\beta\beta}
-\frac{2u^2_{1}}{v^2}\sum_{\beta=2}^n(d_1\cos\theta )_\beta u_{1\beta}\\&
+\big[\frac{2}{v^2}(d_1\cos\theta )_1-\frac{2u_{1}w}{v^2}\alpha_0d_1\big]\big[u_t-Av\big]+\frac{1}{v}\big[Av-u_t\big]^2\\&
+\frac{2\alpha_0u_{1}w }{S_1v^2}\sum_{\beta=2}^nd_\beta \big[u_1(d_{1}\cos\theta)_\beta +w\alpha_0 d_\beta\big].
\end{split}\end{equation}
Inserting \eqref{4aijwij-wt-3-1} into \eqref{4aijwij-wt-3}, we have
\begin{equation}\label{4aijwij-wt-4}\begin{split}
a^{ij}w_{ij}-w_t
=&\frac{1}{v}\big[\sum_{\beta=2}^nu_{\beta\beta}\big]^2+\frac{1}{v}\sum_{\beta=2}^nu^2_{\beta \beta}+2\sum_{\beta=2}^n\big[(d_\beta\cos\theta )_\beta
+\frac{\alpha_0u_{1}w }{S_1v^2}d_\beta S_\beta\big] u_{\beta\beta}\\&+\big[2A-\frac{2u_t}{v}+\frac{2u_{1}w}{v^2}\alpha_0d_1-\frac{2}{v^2}(d_1\cos\theta )_1\big]\sum_{\beta=2}^nu_{\beta\beta}\\&+\frac{2}{v^3}\sum_{\beta=2}^nu^2_{1\beta }
+\frac{2}{v^2}\sum_{\beta=2}^n\big[(d_\beta\cos\theta )_1+(d_1\cos\theta )_\beta\big]u_{1\beta}
+\frac{1}{v}\big[Av-u_t\big]^2\\&+\big[\frac{2}{v^2}(d_1\cos\theta )_1-\frac{2u_{1}w}{v^2}\alpha_0d_1\big]\big[u_t-Av\big]
+\frac{Au_{1}}{v}\big[u_1(d_1\cos\theta )_1+w\alpha_0 d_1\big]\\&
+\frac{2\alpha_0u_{1}w }{S_1v^2}\sum_{\beta=2}^nd_\beta \big[u_1(d_{1}\cos\theta)_\beta +w\alpha_0 d_\beta\big]
+u_1a^{ij}(d_1\cos\theta )_{ij}.
\end{split}\end{equation}
Then inserting \eqref{4aijwij-wt-4} into \eqref{4aijvarphiij-varphit-1}, we have
\begin{equation}\label{4aijvarphiij-varphit-2}
\begin{split}
0 \ge& a^{ij}\varphi_{ij}-\varphi_t\\
\ge&\frac{1}{w}\bigg\{\frac{1}{v}\big[\sum_{\beta=2}^nu_{\beta\beta}\big]^2+\frac{1}{v}\sum_{\beta=2}^nu^2_{\beta \beta}+2\sum_{\beta=2}^n\big[(d_\beta\cos\theta )_\beta+\frac{\alpha_0u_{1}w }{S_1v^2}d_\beta S_\beta\big] u_{\beta\beta}\\&
\qquad+\big[2A-\frac{2u_t}{v}+\frac{2u_{1}w}{v^2}\alpha_0d_1-\frac{2}{v^2}(d_1\cos\theta )_1\big]\sum_{\beta=2}^nu_{\beta\beta}+\frac{2}{v^3}\sum_{\beta=2}^nu^2_{1\beta }
\\&
\qquad+\frac{2}{v^2}\sum_{\beta=2}^n\big[(d_\beta\cos\theta )_1+(d_1\cos\theta )_\beta\big]u_{1\beta}
+\frac{1}{v}\big[Av-u_t\big]^2\\&
\qquad+\big[\frac{2}{v^2}(d_1\cos\theta )_1-\frac{2u_{1}w}{v^2}\alpha_0d_1\big]\big[u_t-Av\big]
+\frac{Au_{1}}{v}\big[u_1(d_1\cos\theta )_1+w\alpha_0 d_1\big]\\&
\qquad+\frac{2\alpha_0u_{1}w }{S_1v^2}\sum_{\beta=2}^nd_\beta \big[u_1(d_{1}\cos\theta)_\beta +w\alpha_0 d_\beta\big]
+u_1a^{ij}(d_1\cos\theta )_{ij}+\alpha_0 w a^{ij}d_{ij}\\&
\qquad-\alpha_0^2wa^{ij}d_id_j\bigg\}\\
:=&\frac{1}{w}\bigg\{\mathcal  I+\mathcal { II}+\mathcal  {III}\bigg\},
\end{split}
\end{equation}
where
\begin{equation}\label{4I-1}\begin{split}
\mathcal  I :=&\frac{1}{v}\big[\sum_{\beta=2}^nu_{\beta\beta}\big]^2+\frac{1}{v}\sum_{\beta=2}^nu^2_{\beta \beta}+ 2\sum_{\beta=2}^nb_{\beta}u_{\beta\beta},
\end{split}\end{equation}
and
\begin{equation}\label{4bbeta}\begin{split}
b_\beta=&\big[A-\frac{u_t}{v}+\frac{u_{1}w}{v^2}\alpha_0d_1-\frac{1}{v^2}(d_1\cos\theta )_1\big]
+\big[(d_\beta\cos\theta )_\beta+\frac{\alpha_0u_{1}w }{S_1v^2}d_\beta S_\beta\big];
\end{split}\end{equation}
\begin{equation}\label{4II-1}\begin{split}
\mathcal { II} :=&\frac{2}{v^3}\sum_{\beta=2}^nu^2_{1\beta }+\frac{2}{v^2}\sum_{\beta=2}^n\big[(d_\beta\cos\theta )_1+(d_1\cos\theta )_\beta\big]u_{1\beta}\\
\ge&-\frac{1}{2v}\sum_{\beta=2}^n[(d_\beta\cos\theta )_1+(d_1\cos\theta )_\beta]^2\\
\ge &-C(n, |D^2d|,\|\theta\|_{C^2(\overline{\Omega})})v^{-1};
\end{split}\end{equation}
and
\begin{equation}\label{4III-1}\begin{split}
\mathcal  {III} :=&\frac{1}{v}\big[Av-u_t\big]^2+\big[\frac{2}{v^2}(d_1\cos\theta )_1-\frac{2u_{1}w}{v^2}\alpha_0d_1\big]\big[u_t-Av\big]\\&
+\frac{Au_{1}}{v}\big[u_1(d_1\cos\theta )_1- w\alpha_0 d_1\big]
+\frac{2\alpha_0u_{1}w }{S_1v^2}\sum_{\beta=2}^nd_\beta \big[u_1(d_{1}\cos\theta)_\beta +w\alpha_0 d_\beta\big]\\&
+u_1a^{ij}(d_1\cos\theta )_{ij}+\alpha_0 w a^{ij}d_{ij}- \alpha_0^2wa^{ij}d_id_j\\
\ge&\frac{1}{v}\big[Av-u_t\big]^2+\big[\frac{2}{v^2}(d_1\cos\theta )_1-\frac{2u_{1}w}{v^2}\alpha_0d_1\big]\big[u_t-Av\big]\\&
+\frac{Au_{1}}{v}\big[u_1(d_1\cos\theta )_1+w\alpha_0 d_1\big]-C(n, S_0 ,|D^3d|, \|\theta\|_{C^2(\overline{\Omega})})v\\
\ge&A^2v+A\big[\frac{u^2_{1}-2}{v}(d_1\cos\theta )_1+\frac{3u_{1}w}{v}\alpha_0 d_1-2u_t\big]\\&-C(n, S_0 ,u_0,|D^3d|,\|\theta\|_{C^2(\overline{\Omega})})v.
\end{split}\end{equation}

Next we will estimate the term $\mathcal  I$. It is easy to see that $\mathcal  I$ is a quadratic polynomial of $u_{\beta\beta},\,\,2\le\beta\le n$.
 We will calculate the minimum of  $\mathcal I$.
Consider the function $\tilde{I}:\mathbb R^{n-1}\rightarrow \mathbb R$ defined by
\begin{equation}\label{Ip}
\tilde{I}(p)=\frac{1}{v}\big(\sum_{i=2}^np_i\big)^2+\frac1v\sum_{i=2}^np_i^2+2\sum_{i=2}^n b_ip_i,\quad p=(p_2,\cdots,p_n)\in\mathbb R^{n-1}.\end{equation}
At the minimum point of $I$, we have
\begin{equation}0=\frac{\partial \tilde{I}}{\partial p_i}=\frac 2v\sum_{j=2}^np_j+\frac2vp_i+2b_i,\quad 2\leq i\leq n.\end{equation}
It follows that
\begin{align}\label{pj}
\sum_{j=2}^np_j=&-\frac vn\sum_{j=2}^nb_j,\\
p_i=&-\sum_{j=2}^np_j-vb_i=\frac vn\sum_{j=2}^nb_j-vb_i,\quad 2\leq i\leq n.\label{pi}
\end{align}
So inserting  \eqref{pj} and \eqref{pi} into \eqref{Ip}, we have
\begin{equation}\label{2023eq1}\begin{aligned}
\tilde{I}_{\mathrm{min}}=&\frac{1}{v}\big[\frac{v}{n} \sum_{\beta=2}^nb_{\beta}\big]^2
+\frac{1}{v}\sum_{\beta=2}^n\big[\frac{v}{n} \sum_{\beta=2}^nb_{\beta}-vb_{\beta}\big]^2
+\sum_{\beta=2}^n2b_{\beta}\big[\frac{v}{n} \sum_{\beta=2}^nb_{\beta}-vb_{\beta}\big]\\
=&\frac{v}{n}\big[ \sum_{\beta=2}^nb_{\beta}\big]^2-v\sum_{\beta=2}^nb_{\beta}^2.
\end{aligned}\end{equation}
From \eqref{4bbeta}, we have
\begin{align}\label{4bbeta-1}\begin{split}
\sum_{\beta=2}^nb_\beta
=&(n-1)\big[A-\frac{u_t}{v}+\frac{u_{1}w}{v^2}\alpha_0d_1-\frac{(d_1\cos\theta )_1}{v^2}\big]+\big[\sum_{\beta=2}^n(d_\beta\cos\theta )_\beta+\frac{\alpha_0u_{1}w }{S_1v^2}\sum_{\beta=2}^nd_\beta S_\beta\big].
\end{split}\end{align}
Then  combining \eqref{4bbeta}, \eqref{4bbeta-1} and \eqref{2023eq1}, it follows that
\begin{equation}\label{4I-2}\begin{split}
\mathcal I_{\min}
=&-\frac{(n-1)v}{n}[A-\frac{u_t}{v}+\frac{u_{1}w}{v^2}\alpha_0d_1-\frac{1}{v^2}(d_1\cos\theta )_1]^2\\&
-\frac{2v}{n}[A-\frac{u_t}{v}+\frac{u_{1}w}{v^2}\alpha_0d_1-\frac{1}{v^2}(d_1\cos\theta )_1]
\sum_{\beta=2}^n[(d_\beta\cos\theta )_\beta+\frac{\alpha_0u_{1}w }{S_1v^2}d_\beta S_\beta]\\&
+\frac{v}{n}[\sum_{\beta=2}^n\big((d_\beta\cos\theta )_\beta+\frac{\alpha_0u_{1}w }{S_1v^2}d_\beta S_\beta\big)]^2
-v\sum_{\beta=2}^n[(d_\beta\cos\theta )_\beta+\frac{\alpha_0u_{1}w }{S_1v^2}d_\beta S_\beta]^2\\
\ge&-\frac{(n-1)v}{n}A^2-\frac{2A(n-1)v}{n}\big[\frac{u_{1}w}{v^2}\alpha_0d_1-\frac{1}{v^2}(d_1\cos\theta )_1-\frac{u_t}{v}\big]\\&
-\frac{2Av}{n}\sum_{\beta=2}^n\big[(d_\beta\cos\theta )_\beta+\frac{\alpha_0u_{1}w }{S_1v^2}d_\beta S_\beta\big]
-C(n, u_0,  S_0 ,|D^2d|,\|\theta\|_{C^2(\overline{\Omega})})v.
\end{split}\end{equation}
Inserting \eqref{4I-2}, \eqref{4II-1} and \eqref{4III-1} into \eqref{4aijvarphiij-varphit-2}, we have
\begin{equation}\label{4aijvarphiij-varphit-3}
\begin{split}
0 \ge& a^{ij}\varphi_{ij}-\varphi_t\\
\ge&\frac{1}{w}\bigg\{-\frac{(n-1)v}{n}A^2-\frac{2A(n-1)v}{n}\big[\frac{u_{1}w}{v^2}\alpha_0d_1-\frac{1}{v^2}(d_1\cos\theta )_1-\frac{u_t}{v}\big]\\&
\qquad-\frac{2Av}{n}\sum_{\beta=2}^n\big[(d_\beta\cos\theta )_\beta+\frac{\alpha_0u_{1}w }{S_1v^2}d_\beta S_\beta\big]+A^2v\\&\qquad
+A\big[\frac{u^2_{1}-2}{v}(d_1\cos\theta )_1+\frac{u_{1}w}{v}\alpha_0 d_1-2u_t\big]\\&\qquad
-C(n,u_0, S_0 ,|D^3d|, \|\theta\|_{C^2(\overline{\Omega})})v-C(n, |D^2d|, \|\theta\|_{C^2(\overline{\Omega})})v^{-1}\bigg\}\\
\ge
&\frac{1}{w}\bigg\{\frac{A^2v}{n}
-\frac{2Av}{n}\sum_{\beta=2}^n\big[(d_\beta\cos\theta )_\beta+\frac{\alpha_0u_{1}w }{S_1v^2}d_\beta S_\beta\big]\\&\qquad
+A\big[(\frac{u^2_{1}}{v}-\frac{2}{nv})(d_1\cos\theta )_1+\frac{n+2}{n}\frac{u_{1}w}{v}\alpha_0 d_1-\frac{2}{n}u_t\big]\\&\qquad-C(n,u_0, S_0 ,|D^3d|, \|\theta\|_{C^2(\overline{\Omega})})v-C(n, |D^2d|, \|\theta\|_{C^2(\overline{\Omega})})v^{-1}\bigg\}\\
\ge&\frac{1}{w}\bigg\{\frac{A^2v}{2n}
-C_7(n,u_0, S_0 ,|D^3d|,\|\theta\|_{C^2(\overline{\Omega})})v-C_8(n, |D^2d|,\|\theta\|_{C^2(\overline{\Omega})})\bigg\}\\
\ge&\frac{1}{w}\bigg\{\frac{A^2v}{4n}
-C_9(n, |D^2d|,\|\theta\|_{C^2(\overline{\Omega})})\bigg\},
\end{split}
\end{equation}
where we have chosen $\frac{A^2}{4n}= C_7 (n,u_0, S_0 ,|D^3d|,\|\theta\|_{C^2(\overline{\Omega})})$ and have used the inequality $(1- S_0 )v\le w\le 2v$.
Thus we obtain the gradient bound.

Thus we complete the proof of Theorem~\ref{Thm1.3}.
\end{proof}

{\bf Acknowledgement}. The authors would like to thank Professor Xi-Nan Ma for valuable discussions and also express  their sincere thanks to the referee for his/her
 careful review and valuable comments.

{\bf No conflict of interest statement}. On behalf of all authors, the corresponding author states that there is no conflict of interest.

{\bf Data availability statement}. All data generated or analysed during this study are included in this article.

\end{document}